\documentclass[12pt,twoside]{article}

\usepackage{fancyhdr}
\usepackage{amsmath}
\usepackage{amscd}
\usepackage{amssymb}
\usepackage{amsfonts}
\usepackage{euscript}
\usepackage{oldgerm}
\usepackage{calligra}
\usepackage{mathrsfs}
\usepackage{hyperref}
\usepackage{url}
\usepackage{lastpage}
\usepackage{multirow}
\usepackage{graphicx}
\usepackage{stmaryrd}
\usepackage{wasysym}
\usepackage{pifont}
\usepackage{diagrams}
\usepackage{draftwatermark}
\SetWatermarkAngle{45}
\SetWatermarkLightness{0.9}
\SetWatermarkFontSize{5cm}
\SetWatermarkScale{2}
\SetWatermarkText{ }

\renewcommand{\thefootnote}{\fnsymbol{footnote}}

\long\def\sfootnote[#1]#2{\begingroup
\def\thefootnote{\fnsymbol{footnote}}\footnote[#1]{#2}\endgroup}

\newtheorem{theorem}{Theorem}[section]
\newtheorem{definition}[theorem]{Definition}

\newtheorem{proposition}[theorem]{Proposition}

\newenvironment{proof}{\noindent\mbox{\bf Proof.}}
{\hfill\mbox{\ding{111}}\bigskip}

\begin{document}

\pagestyle{fancy}
\lhead[page \thepage \ (of \pageref{LastPage})]{}
\chead[{\bf  Theoremizing  Yablo's Paradox}]{{\bf  Theoremizing  Yablo's Paradox}}
\rhead[]{page \thepage \ (of \pageref{LastPage})}
\lfoot[\copyright\ {\sf Ahmad Karimi \& Saeed Salehi 2014}]
{$\varoint^{\Sigma\alpha\epsilon\epsilon\partial}_{\Sigma\alpha\ell\epsilon\hslash\imath}
\centerdot${\footnotesize {\rm ir}}}
\cfoot[{\footnotesize {\tt  }}]{{\footnotesize {\tt  }}}
\rfoot[$\varoint^{\Sigma\alpha\epsilon\epsilon\partial}_{\Sigma\alpha\ell\epsilon\hslash\imath}\centerdot$
{\footnotesize {\rm ir}}]{\copyright\ {\sf Ahmad Karimi \& Saeed Salehi 2014}}
\renewcommand{\headrulewidth}{1pt}
\renewcommand{\footrulewidth}{1pt}
\thispagestyle{empty}

\begin{center}
\begin{table}
\hspace{0.75em}
\begin{tabular}{| c | l  || l | c |}
\hline
 \multirow{7}{*}{\includegraphics[scale=0.65]{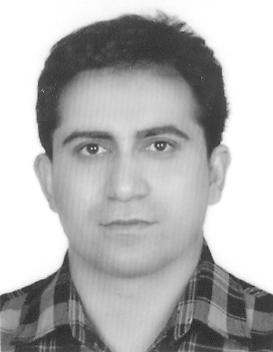}}&    &  &
 \multirow{7}{*}{ \ \ } \ \ \ \ \\
 &     \ \ {\large{\sc Ahmad Karimi}}  \ \  \ & \ \    Tel: \, +98 (0)919 510 2790     &  \\
 &   \ \ Department of  Mathematics \ \  \ & \ \ Fax: \ +98 (0)21 8288 3493    & \\
 &   \ \ Tarbiat Modares University  \ \ \  & \ \ E-mail: \!\!{\tt  a.karimi40}{\sf @}{\tt  yahoo.com}   &  \\
 &  \ \ P.O.Box  14115--134  \ \ \ &   \ \  Behbahan KA Univ. of  Tech.    &  \\
 &   \ \ Tehran, IRAN \ \ \ &  \ \   61635--151 Behbahan, IRAN &   \\
 &    &  &  \\
 \hline
\hline
 \multirow{7}{*}{\includegraphics[scale=0.4]{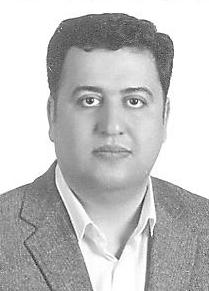}}&    &  &
 \multirow{7}{*}{${\huge \varoint^{\Sigma\alpha\epsilon\epsilon\partial}_{\Sigma\alpha\ell\epsilon\hslash\imath}\centerdot}${{\rm ir}}} \\
 &     \ \ {\large{\sc Saeed Salehi}}  \ \  \ & \ \    Tel: \, +98 (0)411 339 2905    &  \\
 &   \ \ Department of Mathematics \ \  \ & \ \ Fax: \ +98 (0)411 334 2102   & \\
 &   \ \ University of Tabriz \ \ \  & \ \ E-mail: \!\!{\tt /root}{\sf @}{\tt SaeedSalehi.ir/}   &  \\
 &  \ \ P.O.Box 51666--17766 \ \ \ &   \ \ \ \ {\tt /SalehiPour}{\sf @}{\tt TabrizU.ac.ir/}   &  \\
 &   \ \ Tabriz, IRAN \ \ \ & \ \ Web: \  \ {\tt http:\!/\!/SaeedSalehi.ir/}  &  \\
 &    &  &  \\
 \hline
\end{tabular}
\end{table}
\end{center}

\vspace{1.5em}

\begin{center}
{\bf {\Large Theoremizing  Yablo's Paradox}}
\end{center}

\vspace{1.5em}

\begin{abstract}
To counter a general belief that all the paradoxes stem from a kind
of circularity (or involve some self--reference, or use a diagonal
argument) Stephen Yablo designed a paradox in 1993 that seemingly
avoided self--reference. We turn Yablo's paradox, the most
challenging paradox in the recent years,  into a genuine mathematical theorem in
Linear Temporal  Logic (LTL). Indeed, Yablo's paradox comes in several varieties; and he showed in 2004 that
there are other versions   that are equally paradoxical.
Formalizing these versions of Yablo's paradox, we prove some theorems in LTL.
This is the  first time that Yablo's paradox(es)
become  new(ly discovered) theorems in mathematics and logic.
\bigskip

\bigskip


\centerline{${\backsim\!\backsim\!\backsim\!\backsim\!\backsim\!\backsim\!\backsim\!
\backsim\!\backsim\!\backsim\!\backsim\!\backsim\!\backsim\!\backsim\!
\backsim\!\backsim\!\backsim\!\backsim\!\backsim\!\backsim\!\backsim\!
\backsim\!\backsim\!\backsim\!\backsim\!\backsim\!\backsim\!\backsim\!
\backsim\!\backsim\!\backsim\!\backsim\!\backsim\!\backsim\!\backsim\!
\backsim\!\backsim\!\backsim\!\backsim\!\backsim\!\backsim\!\backsim\!
\backsim\!\backsim\!\backsim\!\backsim\!\backsim\!\backsim\!\backsim\!
\backsim\!\backsim\!\backsim\!\backsim\!\backsim\!\backsim\!\backsim\!
\backsim\!\backsim\!\backsim\!\backsim\!\backsim\!\backsim\!\backsim\!
\backsim\!\backsim\!\backsim\!\backsim\!\backsim\!\backsim\!\backsim\!
\backsim\!\backsim\!\backsim}$}

\bigskip


\end{abstract}

\bigskip

\medskip

\vfill

\hspace{.75em} \fbox{\textsl{\footnotesize Date: 01.06.14  (01 June 2014)}}

\vfill

\bigskip
\noindent\underline{\centerline{}}
\centerline{page 1 (of \pageref{LastPage})}


\newpage
\setcounter{page}{2}
\SetWatermarkAngle{65}
\SetWatermarkLightness{0.925}
\SetWatermarkScale{2.7}
\SetWatermarkText{\!\!\!\!\!\!\!\!\!\!\!\!\!\!\!\!\!\!\!\!\!\!
{\sc \,  Copyrighted Manuscript}}



\section{Introduction}\label{intro}
Paradoxes are interesting puzzles  in philosophy and mathematics.
They can be more  interesting when they turn into genuine theorems.
For example, Russell's paradox  which collapsed Frege's foundations
of mathematics,  is now a classical theorem in set theory, implying
that no set of all sets can exist. Or, as another example, the Liar
paradox has turned  into Tarski's theorem on the  undefinability of
truth in sufficiently rich languages. This  paradox also appears
implicitly in the proof of G\"odel's first incompleteness theorem.
For this particular theorem, some other paradoxes such as Berry's
(\cite{boolos89, chaitin74}) or Yablo's (\cite{yab1,yab2}) have been used to give
alternative proofs (\cite{cieslinski2,leach13}). A more recent example  is the
surprise examination paradox \cite{chow98} that has turned into a beautiful proof for G\"odel's second
incompleteness theorem (\cite{kritchman10}).

In this paper we transform Yablo's paradox into a theorem
in the Linear Temporal Logic. This paradox, which is the first one
of its kind that supposedly avoids self--reference and circularity
  has been
used for proving an old theorem (\cite{cieslinski2,leach13}) but not a new theorem had
been made out of it. In this paper, for the very first time, we use
this paradox  (actually its argument)   for proving some genuine mathematical  theorem in
Linear Temporal Logic. Roughly speaking, we show that certain
operators do not have fixed--points in this logic, where the proof
is exactly Yablo's paradox (reaching to a contradiction by assuming
the existence of certain fixed--point sentences). Let us note that very many
other operations in the Linear Temporal Logic do have fixed--points,
which constitute some other genuine mathematical theorems.

\section{Yablo's Paradox}\label{sect2}
To counter a general belief that all the paradoxes stem from a kind
of circularity (or involve some self--reference, or use a diagonal
argument) Stephen Yablo designed a paradox in 1993 that seemingly
avoided self--reference (\cite{yab2, yab1}).
Let us  fix  our reading of Yablo's Paradox: Consider the sequence
of sentences $\{\mathcal{Y}_n\}_{n\in\mathbb{N}}$ such that for each
$n\in\mathbb{N}$: \qquad\qquad {$\mathcal{Y}_n \Longleftrightarrow
\forall k>n\ (\mathcal{Y}_k$ is not true).}

The paradox follows from the following deductions. For each
$n\in\mathbb{N}$,

\begin{tabular}{ccl}
\qquad $\mathcal{Y}_n$ & $\Longrightarrow$ &  $\forall k>n\ (\mathcal{Y}_k$ is not true)\\
\qquad  & $\Longrightarrow$ & $(\mathcal{Y}_{n+1}$ is not true)  and
 $\forall k>n+1\ (\mathcal{Y}_k$ is not true) \\
\qquad  & $\Longrightarrow$ & $(\mathcal{Y}_{n+1}$ is not true)  and
 $(\mathcal{Y}_{n+1}$ is true),  \\
\end{tabular}

\noindent thus $\mathcal{Y}_n$ is not true. So,
{$\forall k\ (\mathcal{Y}_k$ is not true),}
and in particular {$\forall k>0\ (\mathcal{Y}_k$
is not true),} and so $\mathcal{Y}_0$ must be true (and not true at
the same time); contradiction!

Some paradoxes turn  into mathematical--logical tautologies and so
become (interesting) theorems. For example, Liar's paradox when
translated into first--order logic is a sentence $L$ such that
$L\leftrightarrow\neg L$. The fact that this is contradictory is
equivalent to the fact that the formula
$\neg\big(\varphi\leftrightarrow\neg\varphi\big)$ is a tautology in
propositional logic. As another less trivial paradox, take Russell's
paradox: there can be no set $S$  such that for every $x$ we have
$x\in S\leftrightarrow x\not\in x$. Writing this in first--order
logic (in the language $\{\in\}$) we have a logical theorem:
$\neg\exists y\forall x (x\in y\leftrightarrow x\not\in x)$. Indeed, this first--order logical tautology still
holds when we replace the membership relation $\in$ with an
arbitrary binary relation $R$: the sentence
$\neg\exists y \forall x (xRy\leftrightarrow \neg xRx)$ is again a
first--order logical tautology. On the other hand if $xRy$ is
interpreted as ``$y$ shaves $x$'' then the above tautology is nothing
but Barber's  Paradox.   As for Yablo's paradox, J.~Ketland
 has translated it into first--order  logic   (called Uniform
 Homogeneous Yablo Scheme) in \cite{ket05}:
\newline\centerline{$({\sf Y}):\;\forall x\big(\varphi(x)\leftrightarrow\forall
 y[xRy\rightarrow\neg\varphi(y)]\big)$,}
 \noindent where $R$ is a binary
 formula (which could be a binary relation symbol, i.e. an atomic
 formula) with the auxiliary axioms stating that $R$ is total and
 transitive:
 \newline\centerline{ $({\sf A_1}):\;\forall x\exists y(xRy)$ \  and \quad
 $({\sf A_2}):\;\forall x,y,z (xRyRz\rightarrow xRz)$.}
 \noindent A Yablo-like argument can show that
  the formula  $\neg({\sf Y}\wedge{\sf A_1}\wedge{\sf A_2})$
 is a first--order tautology.

\section{Linear Temporal Logic}


Here, we show that there is another way to have a formal version of
Yablo's paradox (different from the formalized version discussed
above), and that is in Linear Temporal Logic. The (propositional)
linear temporal logic (LTL) is a logical formalism that can refer
to time; in LTL  one can encode formulae about the future, e.g., a
condition will eventually be true, a condition will be true until
another fact becomes true, etc. LTL was first proposed for the
formal verification of computer programs in 1977 by Amir Pnueli
\cite{pnueli77}. For a modern introduction to LTL and its syntax and
semantics see e.g. \cite{temporal}. Two modality operators in LTL that we
will use are the ``next'' modality denoted by $\Circle$ and the ``always''
modality denoted as $\Box$.

\subsection{Syntax and Semantics of LTL}
We assume the reader is familiar with the general framework of LTL, but for the sake of accessibility, we list the main notations, definitions and theorems which will be referred to later on. For details we refer the reader to \cite{temporal}.   Let $\textbf{V}$ be a set of \textit{propositional constants}. The alphabet of a basic language
${\cal{L}}_{\rm LTL}(\textbf{V})$ (also shortly: ${\cal{L}}_{\rm LTL}$) of propositional linear temporal logic LTL is given by

\quad all propositional constants of $\textbf{V}$ and the symbols $\{\textbf{false},\rightarrow,\Circle,\Box,(,)\}$.

\noindent The inductive definition of formulas (of ${\cal{L}}_{\rm LTL}(\textbf{V})$) is as follows: 

1---Every propositional constant of $\textbf{V}$ and also the constant symbol $\textbf{false}$ is a formula.

2---If $\varphi$ and $\psi$ are formulas then $(\varphi \rightarrow \psi)$ is a formula.

3---If $\varphi$ is a formula then $\Circle \varphi$ and $\Box \varphi$ are formulas.

\noindent Further operators can be introduced as abbreviations:
\newline\centerline{$\neg,  \vee, \wedge, \leftrightarrow,  \textbf{true}$ as in classical logic, and
$\diamondsuit \varphi \equiv \neg\Box\neg \varphi$.}
The temporal operators $\Circle, \Box,$ and $\diamondsuit$ are called \textit{next time}, \textit{always (or henceforth)}, and
\textit{sometime (or eventuality)} operators, respectively. Formulas $\Circle \varphi$, $\Box \varphi$, and $\diamondsuit \varphi$ are
typically read ``next $\varphi$'', ``always $\varphi$'', and ``sometime $\varphi$''.

Semantical interpretations in classical propositional logic are given by Boolean
valuations. For LTL we have to extend this concept according to our informal idea
that formulas are evaluated over sequences of states (time scales). Let $\textbf{V}$ be a set of propositional constants. A \textit{temporal (or Kripke) structure} for
$\textbf{V}$ is an infinite sequence ${\cal{K}} = (\eta_0,\eta_1,\eta_2,...)$ of mappings
$\eta_i : \textbf{V} \rightarrow \{\frak{ff}, \frak{tt}\}$
called \textit{states}, and  $\eta_0$ is called the \textit{initial state} of ${\cal{K}}$. Observe that states are just valuations in
the classical logic sense. For ${\cal{K}}$ and $i\in {\cal{K}}$, we define ${\cal{K}}_i(F)\in \{\frak{ff},\frak{tt}\}$ (informally
meaning the ``truth value of $F$ in the $i^{\rm th}$ state of ${\cal{K}}$'') for every formula $F$ inductively
as follows:

01.~ ${\cal{K}}_i(v) = \eta_i(v)$ ~~for $v\in \textbf{V}$.

02.~ ${\cal{K}}_i(\textbf{false}) = \frak{ff}$.

03.~ ${\cal{K}}_i(\varphi\rightarrow \psi) = \frak{tt}  \iff {\cal{K}}_i(\varphi)=\frak{ff}$ ~or~ ${\cal{K}}_i(\psi)=\frak{tt}$.

04.~ ${\cal{K}}_i(\Circle \varphi)= {\cal{K}}_{i+1}(\varphi)$.

05.~ ${\cal{K}}_i(\Box \varphi)=\frak{tt} \iff  {\cal{K}}_j(\varphi)=\frak{tt}$ ~for every $j\geq i$.

\noindent Obviously, the formula $\textbf{false}$ and the operator $\rightarrow$ behave classically in each state. The
definitions for $\Circle$ and $\Box$ make these operators formalize the phrases {\it in the next state}
and {\it from this step onward}. More precisely, the formula $\Box \varphi$ informally
means ``$\varphi$ holds in all forthcoming states including the present one''.
The definitions induce the following truth values for the formula abbreviations:

06.~ ${\cal{K}}_i(\neg \varphi)=\frak{tt} \iff {\cal{K}}_i(\varphi)=\frak{ff}$.

07.~ ${\cal{K}}_i(\varphi \vee \psi) =\frak{tt} \iff {\cal{K}}_i(\varphi)=\frak{tt}$ ~or~ ${\cal{K}}_i(\psi)=\frak{tt}$.

08.~ ${\cal{K}}_i(\varphi \wedge \psi)=\frak{tt} \iff {\cal{K}}_i(\varphi)=\frak{tt}$ ~and~ ${\cal{K}}_i(\psi)=\frak{tt}$.

09.~ ${\cal{K}}_i(\varphi\leftrightarrow \psi)=\frak{tt} \iff {\cal{K}}_i(\varphi)={\cal{K}}_i(\psi)$.

10.~ ${\cal{K}}_i(\textbf{true})=\frak{tt}$.

11.~ ${\cal{K}}_i(\diamondsuit \varphi)=\frak{tt} \iff  {\cal{K}}_j(\varphi)=\frak{tt}$ ~for some $j\geq i$.

\begin{definition}[\cite{temporal}] A formula $\varphi$ of ${\cal{L}}_{LTL}(\textbf{V})$ is called \textit{valid} in the temporal structure ${\cal{K}}$ for
$\textbf{V}$ (or ${\cal{K}}$ satisfies $\varphi$), denoted by $\models_{{\cal{K}}}\varphi$, if ${\cal{K}}_i(\varphi)=\frak{tt}$ for every $i\in \Bbb{N}$. The formula $\varphi$ is called a
consequence of a set ${\cal{F}}$ of formulas $({\cal{F}}\models \varphi)$ if $\models_{{\cal{K}}}\varphi$ holds for every ${\cal{K}}$ such that $\models_{{\cal{K}}}\psi$ for all $\psi\in {\cal{F}}$.  The formula $\varphi$ is called \textit{(universally) valid} {\rm (}$\models \varphi${\rm )} if $\emptyset \models \varphi$.
A formula $\varphi$ is called (locally) satisfiable if there is a temporal structure
${\cal{K}}$ and $i\in \Bbb{N}$ such that ${\cal{K}}(\varphi)=\frak{tt}$. 
\hfill $\bigtriangleup\hspace{-0.95em}\blacktriangle$
\end{definition}

The formula $\Circle\varphi$ holds (in the current
moment) when $\varphi$ is true in the ``next step'', and the formula
$\Box\varphi$ is true (in the current moment) when $\varphi$ is true
``now and forever'' (``always in the future''). In the other words,
$\Box$ is the reflexive and transitive closure of $\Circle$. So the
formula $\Circle\Box\psi$ is true when $\psi$ is true from the
next step onward, that is $\psi$ holds in the next step, and the
step after that, and the step after that, etc. The same holds for
$\Box\Circle\psi$; indeed the formula
$\Circle\Box\psi\longleftrightarrow\Box\Circle\psi$ is a law of
$LTL$ (T12 on page 28  of \cite{temporal}). It can also be seen that  the formula $\Circle\neg\varphi\longleftrightarrow\neg\Circle\varphi$ is always true (is a law of LTL, see T1 on page 27 of \cite{temporal}), since $\varphi$ is untrue in the next step if and only if it is not the case that ``$\varphi$ is true in the next step''. Whence, we have the equivalences $\Circle\Box\neg\varphi\longleftrightarrow\Box\Circle\neg\varphi\longleftrightarrow\Box\neg\Circle\varphi$
in LTL. The following theorem will be used in our arguments.

\begin{theorem}[\cite{temporal}]\label{th2.1.9Tempo} {\rm LTL} $\models \varphi$ if and only if $\neg \varphi$ is not satisfiable.
\end{theorem}

\subsection{Paradoxical and Non--Paradoxical  Fixed--Points}\label{}
A version of Yablo's paradox is a sentence $\mathscr{Y}$ that
satisfies the followng equivalences
\newline\centerline{$\mathscr{Y}\!\longleftrightarrow\!\Circle\Box\neg
\mathscr{Y} \quad \big(\!\!\longleftrightarrow\!\Box\Circle\neg\mathscr{Y}\!\longleftrightarrow\!\Box\neg\Circle\mathscr{Y}\big)$}  In
the other words $\mathscr{Y}$ is a fixed--point of the operator
$x\mapsto\Circle\Box\neg x\ \big(\!\!\equiv\Box\Circle\neg x\equiv\Box\neg\Circle x\big)$.
Yablo's argument in his  paradox amounts to showing that this
operator  does not have any fixed--point
in LTL. The semantic proof (i.e. non--existence of any such
fixed--point in  any Kripke model of LTL) is exactly the same as
Yablo's argument. Now, Yablo's paradox becomes the following   theorem.

\begin{theorem}\label{thm-yab1} {\rm LTL} $\models \neg \Box (\varphi\leftrightarrow\Circle\Box\neg\varphi)$.
\end{theorem}

\begin{proof}
To show this formula is valid will exactly follow the line of  Yablo's reasoning to obtain his paradox,  this time in  LTL. By Theorem~\ref{th2.1.9Tempo}, to prove the formula $\neg \Box (\varphi\leftrightarrow\Circle\Box\neg\varphi)$ is valid in LTL, we need to show the formula $\Box (\varphi\leftrightarrow\Circle\Box\neg\varphi)$ is not satisfiable. For a moment assume that there is a Kripke structure ${\cal{K}}$ and $n\in\Bbb{N}$ for which ${\cal{K}}_n\big(\Box(\varphi \leftrightarrow\Circle\Box\neg \varphi)\big) = \frak{tt}$. Then $\forall i\geq n\;{\cal{K}}_i(\varphi \leftrightarrow\Circle\Box\neg \varphi)=\frak{tt}$ which implies that $\forall i\geq n \; {\cal{K}}_i(\varphi) =
 {\cal{K}}_i(\Circle\Box\neg \varphi)= {\cal{K}}_{i+1}(\Box\neg \varphi)$.
We distinguish two cases:

\;(1)\; For some $j\geq n$ we have ${\cal K}_j(\varphi)=\frak{tt}$. Then ${\cal K}_{j+1}(\Box\neg\varphi)=\frak{tt}$   so ${\cal K}_{j+l}(\varphi)=\frak{ff}$ for all $l\geq 1$. In particular ${\cal K}_{j+1}(\varphi)=\frak{ff}$ whence ${\cal{K}}_{j+2}(\Box\neg \varphi)=\frak{ff}$ which is in contradiction with ${\cal K}_{j+1}(\Box\neg\varphi)=\frak{tt}$.

\;(2)\; For all $j\geq n$ we have ${\cal K}_j(\varphi)=\frak{ff}$. So $\frak{ff}={\cal{K}}_n(\varphi) =
  {\cal{K}}_{n+1}(\Box\neg \varphi)$ hence there must exist some $i>n$ with ${\cal K}_i(\varphi)=\frak{tt}$ which contradicts (1) above.

\noindent Thus, the formula $\Box (\varphi\leftrightarrow\Circle\Box\neg\varphi)$ cannot be satisfiable in LTL.
\end{proof}

Also, a G\"odel--like argument can show that the operators $x\mapsto\neg\Box
x$ and $x\mapsto\Box\neg x$ cannot have any fixed--points in LTL
as well.

\begin{proposition}
The operators $x\mapsto\neg\Box x$ and $x\mapsto\Box\neg x$ do not
have any fixed--points in {\rm LTL}; i.e. for any formula $\varphi$ we
have ${\rm LTL}\models\neg\Box(\varphi\leftrightarrow\neg\Box\varphi)$ and
${\rm LTL}\models\neg\Box(\varphi\leftrightarrow\Box\neg\varphi)$.
\end{proposition}
\begin{proof}
We show that satisfiability of $\Box(\varphi\leftrightarrow\Box\neg\varphi)$ in LTL leads to a contradiction. For a  moment let
there exist some Kripke structure ${\cal{K}}$ and $n\in\Bbb{N}$ for which ${\cal{K}}_n(\Box(\varphi\leftrightarrow\Box\neg\varphi))=\frak{tt}$. Then for any $i\geq n$ we have ${\cal{K}}_i(\varphi\leftrightarrow\Box\neg\varphi)=\frak{tt}$ whence $\forall i\geq n \;  {\cal{K}}_i(\varphi) = {\cal{K}}_i(\Box\neg \varphi)$.    This already implies that $\forall i\geq n\; {\cal K}_i(\varphi)=\frak{ff}$ (since $\models\Box\neg\varphi\rightarrow\neg\varphi$). Then, in particular, $\frak{ff}={\cal{K}}_n(\varphi)= {\cal{K}}_n(\Box\neg \varphi)$ and so there must exist some $m\geq n$ such that ${\cal K}_m(\neg\varphi)=\frak{ff}$ contradiction! \end{proof}

\noindent Some  other operators like $x\mapsto \Box x$ or $x\mapsto\neg\Circle
x$ do have fixed--points; $\textbf{true}$ or $\textbf{false}$ for the  former and the
sequences $\langle \frak{ff},\frak{tt},\frak{ff},\frak{tt},\frak{ff},\frak{tt},\cdots\rangle$ or $\langle
\frak{tt},\frak{ff},\frak{tt},\frak{ff},\frak{tt},\frak{ff},\cdots\rangle$ for the latter (see \cite{ber09}).

\section{Other Versions of Yablo's Paradox}
Yablo's paradox comes in several varieties \cite{yab3}; here  we show that other versions of Yablo's paradox   become interesting theorems in LTL as well.

\bigskip 

 \quad (\textsf{always}): \hspace{4em} $\mathcal{Y}_n \iff \forall\,i>n\; (\mathcal{Y}_i \textrm{ is not true })$. 
 
 \smallskip 

\quad  (\textsf{sometimes}): \hspace{2.25em} $\mathcal{Y}_n \iff  \exists\,i>n\; (\mathcal{Y}_i \textrm{ is not true })$.

 \smallskip 
 
\quad  (\textsf{almost always}): \hspace{0.85em} $\mathcal{Y}_n \iff  \exists\,i>n\; \forall j\geq i\; (\mathcal{Y}_i \textrm{ is not true })$.

 \smallskip 
 
 \quad (\textsf{infinitely often}): \hspace{0.6em} $\mathcal{Y}_n \iff  \forall\,i>n\; \exists j\geq i\; (\mathcal{Y}_i \textrm{ is not true })$.

\bigskip 

\noindent It can be seen that all the sequences $\{\mathcal{Y}_n\}_{n\in\mathbb{N}}$ of sentences above are paradoxical.  These sequences of sentences can be formalized in LTL as follows:

\bigskip 

 \quad (\textsf{always}): \hspace{4em} $\mathscr{Y}\!\longleftrightarrow\!\Circle\Box\neg
\mathscr{Y} \quad   \big(\!\!\longleftrightarrow\!\Box\Circle\neg\mathscr{Y}\!\longleftrightarrow\!\Box\neg\Circle\mathscr{Y}\big)$.

\smallskip

\quad  (\textsf{sometimes}): \hspace{2.25em} $\mathscr{Y}\!\longleftrightarrow\!\Circle\diamondsuit\neg
\mathscr{Y} \quad  \big(\!\!\longleftrightarrow\!\diamondsuit\Circle\neg\mathscr{Y}\!\longleftrightarrow\!\diamondsuit
\neg\Circle\mathscr{Y}\big)$.

 \smallskip 

\quad  (\textsf{almost always}): \hspace{0.85em} $\mathscr{Y}\!\longleftrightarrow\!\Circle\diamondsuit\Box\neg
\mathscr{Y} \quad  \big(\!\!\longleftrightarrow\!\diamondsuit\Circle\Box\neg\mathscr{Y}\!\longleftrightarrow
\!\diamondsuit\Box\Circle\neg\mathscr{Y}\!\longleftrightarrow\!\diamondsuit\Box
\neg\Circle\mathscr{Y}\big)$.

 \smallskip 
 
 \quad (\textsf{infinitely often}): \hspace{0.6em} $\mathscr{Y}\!\longleftrightarrow\!\Circle\Box\diamondsuit\neg
\mathscr{Y} \quad  \big(\!\!\longleftrightarrow\!\Box\Circle\diamondsuit\neg\mathscr{Y}\!\longleftrightarrow
\!\Box\diamondsuit\Circle\neg\mathscr{Y}\!\longleftrightarrow\!\Box\diamondsuit\neg\Circle\mathscr{Y}\big)$.

\bigskip 

\noindent  The following (\textsf{sometimes}) counterpart of Theorem~\ref{thm-yab1} directly follows.

\begin{theorem}\label{thm-yab2} {\rm LTL} $\models \neg \Box (\varphi\leftrightarrow\Circle\diamondsuit\neg\varphi)$.
\end{theorem}
\begin{proof}
By Theorem~\ref{thm-yab1} we have {\rm LTL} $\models \neg \Box (\psi\leftrightarrow\Circle\Box\neg\psi)$ for any arbitrary formula $\psi$. In particular for $\psi=\neg\varphi$ we have  {\rm LTL} $\models\neg\Box (\neg\varphi\leftrightarrow\Circle\Box\neg\neg\varphi\leftrightarrow\Circle\neg\diamondsuit\neg\varphi
\leftrightarrow\neg\Circle\diamondsuit\neg\varphi)$, whence for any $\varphi$ we conclude that  {\rm LTL} $\models\neg\Box (\varphi\leftrightarrow\Circle\diamondsuit\neg\varphi)$.
\end{proof}

Let us focus now on the ``\textsf{almost always}'' version of Yablo's paradox. Let $Y_0,Y_1,Y_2,...$ be a sequence of sentences that each sentence, roughly speaking, says ``all sentences, except finitely many, after this sentence are false''. Mathematically, this sequence is as below:

\begin{eqnarray*}
&&Y_0\; : \quad \exists \, i>0 \; \forall j\geq i \; (\mathcal{Y}_j \textrm{ is not true }).\\
&&Y_1\; : \quad \exists \, i>1 \; \forall j\geq i \; (\mathcal{Y}_j \textrm{ is not true }).\\
&&Y_2\; : \quad \exists \, i>2 \; \forall j\geq i \; (\mathcal{Y}_j \textrm{ is not true }).\\
&& \   \vdots \;   \qquad  \qquad \qquad    \qquad \vdots
\end{eqnarray*}

The paradox arises when we try to assign truth values in a consistent way to all   $Y_i$'s. Assume for a moment  that there is a sentence (say) $Y_n$ which is true; so there exists $i>n$ for which all $Y_j$ with $j\geq i$ are untrue. In particular, $Y_{i}$ is  untrue. Since  all the sentences $Y_{i+1}, Y_{i+2},...$ are untrue, so $Y_i$ has to be true. Therefore, $Y_i$ is  true and false the same time, which is a contradiction. Whence, all $Y_n$'s are untrue, so $Y_0$ is true, a contradiction again. Now we turn this version of Yablo's paradox to a theorem in LTL.

\begin{theorem}\label{thm-yab3} {\rm LTL} $\models \neg \Box (\varphi\leftrightarrow\Circle\diamondsuit\Box\neg\varphi)$.
\end{theorem}

\begin{proof}
We show that the formula  $\Box (\varphi\leftrightarrow \Circle\diamondsuit\Box\neg\varphi)$ is not satisfiable in  LTL. For a moment, assume that there is a Kripke structure ${\cal{K}}$ and a state $n\in\Bbb{N}$ for which ${\cal{K}}_n\big(\Box(\varphi\leftrightarrow \Circle\diamondsuit\Box\neg \varphi)\big) = \frak{tt}$. So, we have $\forall i\geq n \; {\cal{K}}_i(\varphi\leftrightarrow \Circle\diamondsuit\Box\neg \varphi)=\frak{tt}$ which implies $\forall i\geq n \; {\cal{K}}_i(\varphi) = {\cal{K}}_i(\Circle\diamondsuit\Box\neg \varphi)$ which is equivalent to $\forall i\geq n \; \exists j\geq0 \; {\cal{K}}_i(\varphi) = {\cal{K}}_{i+j+1}(\Box\neg \varphi)$.

\;(1)\; If there is  some $l\geq n$ such that ${\cal K}_l(\varphi)=\frak{tt}$, then  ${\cal K}_{l+m+1}(\Box\neg\varphi)=\frak{tt}$   for some $m$; so ${\cal K}_{l+m+1}(\varphi)=\frak{ff}$  and also ${\cal K}_{k}(\varphi)=\frak{ff}$ for all $k\geq l+m+1$. On the other hand there must exist some $p\geq 0$ such that ${\cal K}_{l+m+1+p+1}(\Box\neg\varphi)=\frak{ff}$ which implies that ${\cal K}_{l+m+1+p+1+q}(\varphi)=\frak{tt}$ for some $q\geq 0$. This is a contradiction since $l+m+1+p+1+q\geq l+m+1$.

\;(2)\; If  ${\cal K}_l(\varphi)=\frak{ff}$ holds for  all $l\geq n$, then  in particular ${\cal K}_n(\varphi)=\frak{ff}$ and so there exists some $m\geq 0$ such that ${\cal{K}}_{n+m+1}(\Box\neg \varphi)=\frak{ff}$; whence  ${\cal{K}}_{n+m+1+p}(\varphi)=\frak{tt}$ for some $p\geq 0$, which contradicts (1) above.
\end{proof}

Again by the technique of the proof of Theorem~\ref{thm-yab2} we can deduce the following  from Theorem~\ref{thm-yab3}.

\begin{theorem}\label{thm-yab4}  {\rm LTL} $\models \neg \Box (\varphi\leftrightarrow \Circle\Box\diamondsuit \neg\varphi)$.
\end{theorem}

\begin{proof}
{\rm LTL} $\models \neg \Box (\psi\leftrightarrow\Circle\diamondsuit\Box\neg\psi)$ holds for any   formula $\psi$ by Theorem~\ref{thm-yab3}. For $\psi=\neg\varphi$ we obtain  the deduction {\rm LTL} $\models\neg\Box (\neg\varphi\leftrightarrow\Circle\diamondsuit\Box\neg\neg\varphi\leftrightarrow\Circle\neg\Box\diamondsuit\neg\varphi
\leftrightarrow\neg\Circle\Box\diamondsuit\neg\varphi)$ which completes the proof.
\end{proof}





%
%




\end{document}